\documentclass[11pt,a4paper]{article}
\usepackage{amsmath, amsthm, amssymb,amsfonts}
\usepackage[english]{babel}
\usepackage{graphicx}

\theoremstyle{plain}   
\newtheorem{theorem}{Theorem}
\newtheorem{proposition}{Proposition}
\newtheorem{corollary}{Corollary}
\newtheorem{lemma}{Lemma}
\newtheorem{conjecture}{Conjecture}
\newtheorem{remark}{Remark}
\newtheorem{definition}{Definition}
\def\A{\mbox{\boldmath $A$}}

\def\I{\mbox{\boldmath $I$}}
\def\J{\mbox{\boldmath $J$}}

\def\P{\mbox{\boldmath $P$}}

\newcommand{\Tr}{\operatorname{Tr}}\makeatletter
\let\lcm\relax 
\DeclareRobustCommand{\lcm}{\mathop{\operator@font lcm}}

\title{On the nonexistence of almost Moore digraphs with self-repeats}

\author{Arnau Messegu\'e\thanks{Supported in part by grant 2021SGR-00434.}\\
\small Universitat de Lleida\\[-0.8ex]
\small Lleida, Spain\\
\small\tt arnau.messegue@udl.cat 
\and
Josep M. Miret\thanks{Supported in part by grants PID2021-124613OB-I00 and 2021SGR-00434.}\\
\small Universitat de Lleida\\[-0.8ex]
\small Lleida, Spain\\
\small\tt josepmaria.miret@udl.cat
\and
Anita A. Sillasen\\
\small Aalborg University\\[-0.8ex]
\small Aalborg, Denmark\\
\small\tt anitasillasen@gmail.com
}

\begin{document}

\maketitle

\begin{abstract}
An almost Moore digraph is a diregular digraph of degree $d>1$, diameter $k>1$ and order $d+d^2+ \cdots +d^k$. 
 Their existence has only been shown for $k=2$. 
It has also been conjectured that there are no more almost Moore digraphs, but so far their nonexistence has only been proven for $k=3,4$ and for $d=2,3$ when $k\geq 3$.

In this paper we study the structure of the subdigraphs of an almost Moore digraph induced by the vertices fixed by an automorphism determined by a power of the permutation $r$ of repeats of the digraph. We deduce that
each almost Moore digraph of degree $d$ and diameter $k$ with self-repeats has such a subdigraph whose vertices have order $\leq d-1$ under $r$. 
From this, 
we extend the results about the nonexistence of  almost Moore digraphs with self-repeats of degrees 4 and 5 to those whose diameter 
is large enough with respect to the degree. More precisely,  we prove their nonexistence when $k\geq 2(d-1)$ if $k$ is odd and
when $k \geq 2(d-1)^2$ if $k$ is even. We also show that these findings jointly with
other results imply that there are no almost Moore digraphs with self-repeats
for degrees $d$, $6\leq d\leq 12$, and $k>2$.

\end{abstract}


\section{Introduction}
An upper bound for  the largest possible number of vertices of a digraph with maximum out-degree $d$ and diameter $k$ is the Moore bound $1+d+\cdots+d^k$. However,  
Plesn\'{\i}k and  Zn\'{a}m in \cite{PZ74} and later 
Bridges and Toueg in \cite{bridges80} proved the nonexistence of digraphs which reach this bound unless $d = 1$ (directed cycles of length $k+1$) or $k = 1$ (complete symmetric digraphs).
Then, the existence of almost Moore digraphs, also called $(d,k)$-digraphs, that is, digraphs of maximum out-degree $d>1$, diameter $k>1$ and order one less than the Moore bound is an interesting and hard question, which appears in the context of the {degree/diameter problem}. For a survey of this problem, see \cite{ms05}.

Almost Moore digraphs are known to be diregular \cite{diregular}, but concerning their existence it seems they do only exist for $k=2$. In \cite{fiol83} Fiol et al. proved $(d,2)$-digraphs do exist for any degree $d>1$ and in \cite{gimbert01} Gimbert  finalised their complete classification. 

There exist two conjectures such that, assuming that either of them is  true, the nonexistence of $(d,k)$-digraphs for any $k\geq 3$ is proven.  One of them, based on the structure of the out-neighbours \cite{CBU} was given by Cholily in \cite{cholily}. The other, related with the factorization of certain polynomials which appear in the characteristic polinomial of the adjacency matrix, was given by Gimbert \cite{g99} and it is called cyclotomic conjecture. In \cite{cggmm14} the nonexistence is also proven assuming this conjecture. Nevertheless, their nonexistence without any conjecture has only been proven in a few cases.
Miller and Fris \cite{miller92}
proved that there are no $(2,k)$-digraphs for $k\geq 3$ and Baskoro et al. \cite{b95,baskoro973} proved the nonexistence of $(3,k)$-digraphs for $k\geq 3$.
Regarding the diameter, Conde et al. \cite{cggmm08,cggmm13} showed the nonexistence of $(d,3)$ and $(d,4)$-digraphs. More recently, for $(4,k)$ and $(5,k)$-digraphs, in the case of having self-repeats López et al. \cite{almost4-5} proved  they do no exist for $k\geq 5$. In the case without self-repeats, Miret et al. \cite{mrs} have reduced the number of possible structures such digraphs can have.

In this paper, we prove that $(d,k)$-digraphs with self-repeats do not exist  when the diameter $k$ is large with respect to the degree $d$. More precisely, for the pairs $(d,k)$ such that either $k$ is odd with $k\geq 2(d-1)$ or $k$ is even with $k\geq 2(d-1)^2$ there are no $(d,k)$-digraphs with self-repeats.  To do this, we study first the structure of certain subdigraphs of almost Moore digraphs, which are induced by the vertices fixed by an automorphism.  Then, we prove that any $(d,k)$-digraph with self-repeats has such a $(d',k)$-subdigraph, $d'\leq d$, with
permutation cycle structure of the form $(k,\ldots, m_{d'-1},0\ldots,0)$. 
Studying these digraphs in conjunction with well-known results from the literature further properties are deduced. Using these properties together with those obtained when studying the spectra of their adjacency matrix  we can extend the results given in \cite{almost4-5} for degrees 4 and 5 to show there are no almost Moore digraphs with self-repeats when the diameter is large enough with respect to the degree. Finally, combining these results with the validation (verified computationally) that the cyclotomic conjecture holds for certain small values, we conclude that for degree $d$, $6\leq d\leq 12$, there are no almost Moore digraphs with self-repeats.

\section{Almost Moore digraphs}

A digraph $G$ is a finite nonempty set $V(G)$ of objects called {\em vertices} together with a set $E(G)$ of ordered pairs of distinct vertices called {\em arcs}. The {\em order} of $G$ is the cardinality of $V(G)$ denoted by $|V(G)|$. The set of vertices which are adjacent from, respectively to, a given vertex $v$ is denoted by $N^+(v)$, respectively $N^-(v)$, and its cardinality is the {\em out-degree} of $v$, $d^+(v)=|N^+(v)|$, respectively {\em in-degree} of $v$, $d^-(v)=|N^-(v)|$. A digraph is {\em diregular} of degree $d$  if for any vertex $v$, $d^+(v)=d^-(v)=d$. The  length of a shortest walk $u\to v$ is the (walk) {\em distance} from $u$ to $v$. Its maximum  over all pairs of vertices is called the {\em diameter} $k$ of the digraph. Besides, we denote by $N_i^+(v)$, $i\geq 1$, the set of vertices at distance $i$ (in $G$) from $v$. Note that $N^+_1(v)=N^+(v)$. In case, we need to consider  the vertices with these properties but belonging to a given subdigraph $H$ of $G$, it will be denoted by $N_{i,H}^+(v)$.

This paper focuses on almost Moore digraphs (also called $(d,k)$-digraphs of defect 1 or simply $(d,k)$-digraphs), that is, on diregular digraphs of degree $d>1$, diameter $k>1$ and whose order is one less than the Moore bound, that is, 
$$
N=d+d^2+\cdots + d^k.
$$
An almost Moore digraph $G$ of degree $d$ and diameter $k$ has the property that each vertex $v\in V(G)$ has a unique vertex $u\in V(G)$, called the repeat of $v$ and denoted by $r(v)$, such that there are exactly two walks from $v$ to $r(v)$ of length $\leq k$, being at least one of them of length $k$.  If $r(v)=v$, the vertex $v$ is called a self-repeat of $G$. The permutation 
$$
r:V(G) \longrightarrow V(G),
$$
which assigns to each $v\in V(G)$ its repeat $r(v)$, is an automorphism of $G$ (see \cite{baskoro96}). The smallest integer $t\geq 1$ such that $r^t(v)=v$ is called the order of $v$ and it is denote by $ord_r(v).$

Given a $(d,k)$-digraph $G$, its adjacency matrix $\A$ fulfills the equation
\begin{equation}\label{eqn:1}
 \I + \A + \cdots + \A^k = \J + \P
\end{equation}
where $\J$ denotes the all-one matrix and $\P = (p_{ij})$ is the
(0,1)-matrix associated with the permutation $r$ of the set of
vertices $V(G) = \{v_1, \ldots , v_N\}$, which is equivalent to say $p_{ij}$
= $1$ if and only if $r(v_i) = v_j$. Since $G$ has no cycles of length less than $k$, its adjacency matrix $\A$ satisfies
\begin{equation}
\label{traces1k}
\Tr(\A^{\ell})=0, \; \ell=1,2,\ldots,k-1, \mbox{ and } \Tr(\A^k)=\Tr(\P).
\end{equation}
The case $\P=\I$ was studied by Bos\'ak \cite{Bosak} who proved that for $k\geq 3$ there are no $(d,k)$-digraphs. Later, Filipovski and Jajcay studied in \cite{FJ} the $(d,k)$-digraphs with defect $\delta\geq 2$ containing only self-repeat vertices, namely, those whose adjacency matrix satisfies the equation $\I + \A + \cdots + \A^k = \J + \delta\I$, showing their non-existence for $d\geq \delta\geq k+1\geq 4$.

In general, 
for $k\geq 3$, a $(d,k)$-digraph $G$ contains either no cycles of length $k$ or exactly one cycle of length $k$ consisting of self-repeat vertices (see \cite{baskoro96}). Therefore, when $G$ has self-repeats, it has exactly $k$ self-repeats and
\begin{equation}
\label{tracek}
\Tr(\A^k)=k.
\end{equation}
The permutation $r$ has a cycle structure that corresponds to
its unique decomposition in disjoint cycles. The number of permutation
cycles of $G$ of each length  $n \leq N$, will be denoted by $m_{n}$
and the vector
\[
  (m_1,m_2, \ldots ,m_N), \quad \sum_{i=1}^N im_i=N,
\]
will be referred to as the permutation cycle structure of $G$.
Note that in the case $G$ is a $(d,k)$-digraph with self-repeats it turns out $m_1=k$.

\subsection{The cyclotomic conjecture}
\label{sec:conj}

The factorization of the characteristic polynomial of $\J+\P$ in $\mathbb Q[x]$ can be given in terms of the cyclotomic polynomials 
$$
\Phi_i(x)=\prod_{1\leq h\leq \varphi(i)} (x-\zeta_h^i),
$$
where $\zeta_h^i$ ranges over the the primitive roots of unity of order $i$ and where $\varphi(i)$ stands for Euler's function.

From (\ref{eqn:1}), the problem of the factorization of the characteristic polynomial of $\A$ in $\mathbb Q[x]$ is connected to the study of the irreducibility in $\mathbb Q[x]$ of the polynomials 
$$
F_{i,k}(x)=\Phi_i(1+x+\cdots + x^k),
$$
being $k$ the diameter of $G$.

In \cite{g99}, Gimbert proved that $F_{2,k}(x)$ is irreducible and gave for $i>2$ the following conjecture on the irreducibility of these polynomials:

\begin{conjecture}
\label{cycloconj}
Let $i>2$ and $k>1$ be integers. Then
\begin{itemize} 
\item[i)]
If $k$ is even then $F_{i,k}(x)$ is reducible in $\mathbb Q[x]$ if and only if $i \mid (k+2)$, in which case $F_{i,k}(x)$ has just two factors.
\item[ii)] 
If $k$ is odd then $F_{i,k}(x)$ is irreducible in $\mathbb Q[x]$ if and only if $i\mid 2(k+2)$, in which case $F_{i,k}(x)$ has just two factors.
 \end{itemize}
\end{conjecture}

Given a pair of integers $(d,k)$, with $d$ the degree and $k$ the diameter, and a permutation cycle structure $(m_1,m_2,\ldots, m_N)$, if the cyclotomic conjecture were true for the polynomials $F_{i,k}(x)$, for those $i$ such that $m(i)\neq 0$, with $m(i)=\sum_{i|j}m_j$, then $(d,k)$-digraphs with such a permutation cycle structure would not exist \cite{cggmm14}. 

\begin{remark}
\label{rem}
We have checked the factorization of $F_{i,k}(x)$ in $\mathbf{Q}[x]$ for $2<i\leq 14$ and $4<k\leq 200$, verifying the Conjecture \ref{cycloconj} for all these values.
\end{remark}

\section{Subdigraphs of a $(d,k)$-digraph  induced by the fixed vertices of an automorphism}
The main aim of this section is to show that any almost Moore digraph contains a subdigraph which is also an almost Moore digraph having a very specific permutation cycle structure. To do so we analyse the subdigraph of fixed vertices of some automorphism following a very similar line of reasoning as in \cite{fixedpoints}. The first key point in our argument is the following result.

\begin{lemma}
\label{lem:1}
Let $G$ be a $(d,k)$-digraph and let $\varphi$ be an automorphism of $G$. If $u$, $v$ are two distinct fixed vertices by $\varphi$, then any walk of length at most $k$ connecting $u$ with $v$ is fixed by $\varphi^2$.
\end{lemma}

\begin{proof} 
Since $G$ is connected by hypothesis and its diameter is $k$, there exists at least one shortest walk connecting $u$ with $v$ of length at most $k$. Let $\pi$ be such a walk. 
We distinguish two cases: either $v \neq r(u)$ and then such walk is unique, or $v = r(u)$ and then there exists exactly one other shortest walk $\pi'$ of length at most $k$ connecting $u$ with $v$.

In the first case, when $v\neq r(u)$, $\varphi(\pi)$ is a path of length 
equal to the length of $\pi$
and it connects $\varphi(u) = u$ with $\varphi(v) = v$. Therefore,  by the uniqueness of $\pi$, we must have $\varphi(\pi) = \pi$. Hence, $\pi = \varphi^2(\pi) $. 

In the second case, when $v = r(u)$, $\varphi(\pi)$ is a walk of length at most $k$ connecting $\varphi(u) = u$ and $\varphi(v) = v$, so either $\varphi(\pi) = \pi$, and then $\varphi^2(\pi) = \pi$ holds as we wanted to see, or, otherwise, 
$\varphi(\pi) = \pi'$. 
When $\varphi(\pi) = \pi'$, $\pi$ and $\pi'$ must both be of length $k$, and either $\varphi(\pi') = \pi'$ or $\varphi(\pi') = \pi$.
However, it cannot be $\varphi(\pi') = \pi'$ because $\pi' = \varphi(\pi)$, so $\varphi$ would not be a digraph automorphism. Therefore we must have $\varphi(\pi') = \pi$ which implies $\pi = \varphi(\varphi(\pi)) = \varphi^2(\pi)$.  
\end{proof}

\subsection{The subdigraphs $H_{\alpha}$}

We will study  certain subdigraphs of a $(d,k)$-digraph $G$ denoted by $H_{\alpha}$. Indeed, for each positive integer $\alpha >1$ we consider the subdigraph $H_{\alpha}$ of $G$ defined as follows:
\begin{equation}
\label{Halpha}
H_{\alpha}=\{v\in G : \mbox{there exists an integer } t\geq 0 \mbox{ such that } ord_r(v) \mid 2^t \alpha\},
\end{equation}
where $r$ is the permutation of repeats of $G$.

Notice two important facts about the sets $H_{\alpha}$. The first one is that, when $G$ contains self-repeats, the subdigraph of self-repeats, which is a directed cycle $C_k$, is always contained in $H_{\alpha}$ for any  $\alpha$, that is,
$$
C_k\subseteq H_{\alpha}, \quad \forall \alpha >1.
$$
The second important observation is,  since $ord_r(v) = ord_r(r(v))$, that 
\begin{equation}
  \label{ord}
v \in H_{\alpha} \Longleftrightarrow r(v) \in H_{\alpha}.
\end{equation}

Now we introduce the concept of a \emph{$(r,k)$-closed subdigraph} which is a compact notation that is useful later for the following results.

\begin{definition}
A subdigraph $H$ of a $(d,k)$-digraph $G$ is $(r,k)$-closed when for every two distinct vertices $u,v$ from $H$ any shortest path of length at most $k$ connecting $u$ with $v$ is also contained in $H$ and, furthermore, $r(H) \subseteq H$.
\end{definition}

Let $j ,h \geq 1$ be positive integers. Since $r$ is an automorphism,  Lemma \ref{lem:1} implies that any walk connecting two fixed points $u$, $v$ by $r^{2^{j-1}h}$ of length at most $k$ is fixed by $r^{2^jh}$. This property has a key implication: when $j$ runs from $0$ to $N$ the union of all fixed points by $r^{2^jh}$ is a $(r,k)$-closed subdigraph of $G$. 
The next proposition states and proves this result.

\begin{proposition}
\label{kclosed}
Any subdigraph 
$H_{\alpha}$ of a $(d,k)$-digraph $G$ is $(r,k)$-closed.
\end{proposition}
\begin{proof} Let $u,v \in H_{\alpha}$ and let $\pi$ a walk of length at most $k$ connecting $u$ with $v$. By the assumption $u,v\in H_{\alpha}$, we know that there exist non-negative integer values $e_u,e_v$ and  strictly positive integer values $\alpha_u,\alpha_v$ with $\alpha_u, \alpha_v \mid \alpha$ such that: 
$$
ord_r(u) = 2^{e_u}\alpha_u, \qquad 
ord_r(v) = 2^{e_v}\alpha_v.
$$
Then consider the automorphism $\varphi = r^{2^{\max(e_u,e_v)}\alpha}$. 
Of course, $\varphi(u) = u$ and $\varphi(v) = v$ because $ord_r(u), ord_r(v) \mid 2^{\max(e_u,e_v)\alpha}$. By Lemma \ref{lem:1}, for any $w\in \pi$, $w = \varphi^2(w)$. Therefore 
$$
ord_r(w) \mid 2 \cdot 2^{\max(e_u,e_v)}\alpha = 2^{1+ \max(e_u,e_v)}\alpha.
$$ 
This clearly implies $w \in H_{\alpha}$, too.
Finally, if $z \in H_{\alpha}$ then $ord_r(z) \mid 2^t \alpha$ for some integer $t \geq 0$. Then $ord_r(r(z)) \mid 2^t \alpha $ because $ord_r(r(z)) = ord_r(z)$ by (\ref{ord}). From here we get $r(H_{\alpha})  \subseteq H_{\alpha}$ and this completes the proof. 
\end{proof}

\subsection{Regularity of $H_{\alpha}$}

Let $v\in V(G)$ be a vertex of $G$ and let $N^+(v)$ be the set of the exact $d$ neighbours from $v$ such that $(v,v_j) \in E(G)$ for every $j$ with $1 \leq j \leq d$, that is, 
$$N^+(v) = \left\{ v_1,v_2,\ldots,v_d\right\}.
$$ 
We denote by $T(v_j)$ the subset of vertices $w$ from $G$ such that there exists a shortest walk from $v$ to $w$ that goes through $v_j$, with $1\leq j \leq d$. By definition of almost Moore digraph, 
$$T(v_1),T(v_2),\ldots,T(v_d)$$ 
are pairwise disjoint except for at most two indices $i,j$ such that $T(v_i) \cap T(v_j) = \left\{ r(v)\right\}$.  

Now, define for each $j$ with $1 \leq j \leq d$ the following values: 
$$
n(j) = \#\{(z,v) \in E(G) \mid z \in T(v_j)\}.
$$
The distance from $v_j$ to $v$ is either $k-1$ if $v=r(v)$ and $v_j=r(v_j)$ or exactly $k$. Then, for each $j$ with $1 \leq j \leq d$, it holds $n(j) \geq 1$. Furthermore, $n(j) \leq 2$. Otherwise the order of $G$ would be less than or equal $d+d^2+\cdots + d^k-1$. More than this, 
$$n(j) = 2 \Longleftrightarrow v = r(v_j),$$ 
so there is at most one subindex $j$ with $1 \leq j \leq d$ such that  $n(j)=2$. Then write $v_{-j} \in V(G)$ to be the vertex from $T(v_j)$ such that 
$$(v_{-j},v) \in E(G) \text{ when } n(j)=1.$$

And let $v_{-j,1},v_{-j,2}$ be the two distinct vertices from $T(v_j)$ such that 
$$(v_{-j,1},v),(v_{-j,2},v) \in E(G) \text{ when } n(j)=2.$$
We distinguish the following cases:

\begin{center}
\begin{figure}[h!]
  \makebox[\textwidth]{
  \includegraphics[width=0.65\paperwidth]{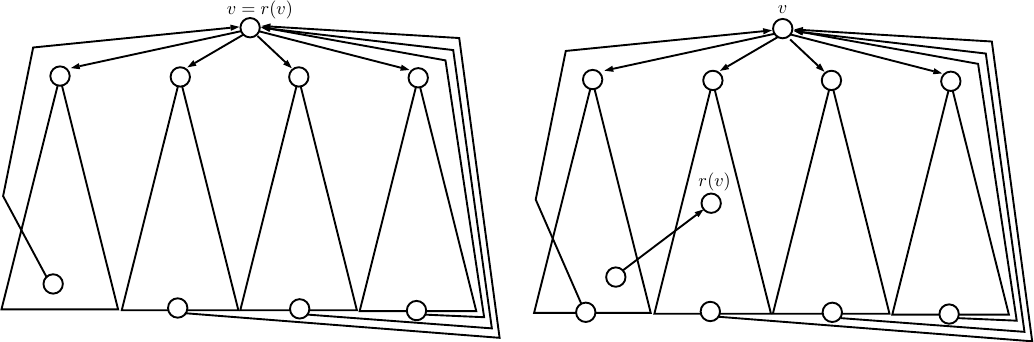}
  }
\caption{The case $n_j(1)$ for all $j$.}
\label{4scenarios-1}
\end{figure}
\end{center}
\begin{itemize}
\item[I)] $n(j)=1$ for all $j$, with $1 \leq j \leq d$. In this case it must happen that $v_{-a} \neq v_{-b}$ for any distinct values $a,b$ with $1 \leq a,b \leq d$, because $G$ is diregular of degree $d$ (see \cite{diregular}). We further distinguish two subcases: 
\begin{itemize}
\item[i)] 
$n(j)=1$ for all $j$, with $1 \leq j \leq d$, and $v$ is a self-repeat (left picture from Figure \ref{4scenarios-1}).
\item[ii)] 
$n(j)=1$ for all $j$, with $1 \leq j \leq d$, and $v$ is not a self-repeat (right picture from Figure~\ref{4scenarios-1}).
\end{itemize}

\item[II)] $n(j)=1$ for all $j$ with $1 \leq j \leq d$ except for one subindex $h$ for which $n(h)=2$. Since $G$ is $d$-regular, in this scenario it can only happen two possibilities: 

\begin{itemize}

\item[i)] $v_{-a} = v_{-b}$ for exactly one pair of distinct values $a,b$ with $1 \leq a,b \leq d$ and $a,b \neq h$  (left picture from Figure \ref{4scenarios-2}).

\item[ii)] $v_{-a} = v_{-h,1}$ or $v_{-a} = v_{-h,2}$ for exactly one value $a$ with $1 \leq a \leq d$ and $a \neq h$  (right picture from Figure \ref{4scenarios-2}).

\begin{center}
\begin{figure}[h!]
  \makebox[\textwidth]{
\includegraphics[width=0.65\paperwidth]{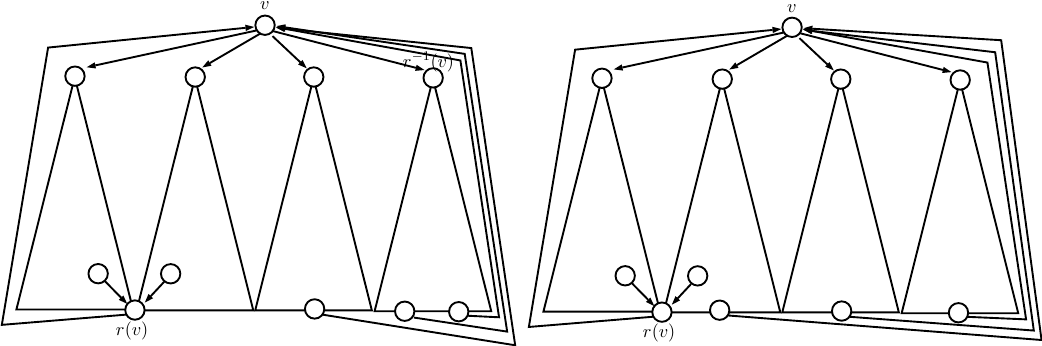}}
\caption{$n(j)=1$ for all $j$ except for on subindex.}
\label{4scenarios-2}
\end{figure}
\end{center}

\end{itemize}
\end{itemize}

As a consequence of this we have the following result.

\begin{corollary}
\label{corol:inout}
Let $H_{\alpha}$ be a subdigraph of a $(d,k)$-digraph $G$ and assume $v \in H_{\alpha}$. If $(z,v) \in E(G)$ is such that $z \in T(v_j)$ with $1 \leq j \leq d$, then $$v_j \in H_{\alpha}\Longleftrightarrow z \in H_{\alpha}.
$$
\end{corollary}

\begin{proof}
By definition, both $v_{j}$ and $z$ belong to $T(v_j)$. Therefore, if $r(v) = v$ and $z \in N_{k-1}^+(v)$ the walk $\pi_j$ connecting $v_j$ with $z$ is of length $k-2$, otherwise it has length $k-1$.  Taking this into account, if $v_j \in H_{\alpha}$, then the walk that starts at $v_j$ follows $\pi_j$ and then uses the arc $(z,v)$ has length at most $(k-1)+1=k$ and goes through $z$. Since $H_{\alpha}$ is $(r,k)$-closed and both $v,v_j$ belong to $H_{\alpha}$ then $z \in H_{\alpha}$, too.  The other way around follows similarly.
\end{proof}

As a direct consequence of this result we obtain:

\begin{corollary}
\label{corol1}
If $v$ is a vertex of $H_{\alpha}$ of a $(d,k)$-digraph $G$, then 
$$d^+_{H_{\alpha}}(v) = d^-_{H_{\alpha}}(v).$$ 
\end{corollary}

\begin{proof} Let $J$ be the collection of subindices $j$ with $1 \leq j \leq d$ such that $v_j \in H_{\alpha} \cap N^+(v)$. Similarly, let $J'$ be the collection of subindices $j$ such that $v_{-j} \in H_{\alpha} \cap N^-(v)$. We will distinguish the former cases:

\begin{itemize}

\item[I)] Case $n(j)=1$ for all $j$: By Corollary \ref{corol:inout} we get $J = J'$ implying our claim.
\item[II] Case $n(j)=1$ for all $j$ except for $h$:
\begin{itemize}
\item[i)]  Since $H_{\alpha}$ is $(r,k)$-closed and $v \in H_{\alpha}$, then $v_{-a} = v_{-b} = r(v) \in H_{\alpha}$ and so $v_a,v_b \in H_{\alpha}$ as well using Corollary \ref{corol:inout}. 
Similarly, $v_h = r^{-1}(v) \in H_{\alpha}$ because, again, $H_{\alpha}$ is $(r,k)$-closed and $v \in H_{\alpha}$, so $v_{-h,1},v_{-h,2} \in H_{\alpha}$ as well using Corollary \ref{corol:inout}. Therefore:
$$
\begin{array}{ll}
N_{H_{\alpha}}^+(v) & = \left\{v_j \mid j \in J\right\} = \left\{v_j \mid j \in J \setminus \left\{a,b,h\right\} \right\} \cup \left\{v_a,v_b,v_h \right\},\\
N_{H_{\alpha}}^-(v) & = \left\{v_{-j} \mid j \in J'\right\}\cup  \left\{v_{-h,1},v_{-h,2}\right\} \\
&= \left\{v_{-j} \mid j \in J' \setminus \left\{a,b\right\} \right\} \cup \left\{v_{-a},v_{-h,1},v_{-h,2} \right\}.
\end{array}
$$
By Corollary \ref{corol:inout} it turns out that $J \setminus \left\{a,b,h\right\} = J' \setminus \left\{ a, b\right\}$. 
\item[ii)]  
Here $v_h = r^{-1}(v) \in H_{\alpha}$ because $H_{\alpha}$ is $(r,k)$-closed and $v \in H_{\alpha}$, so $v_{-h,1},v_{-h,2} \in H_{\alpha}$ as well using Corollary \ref{corol:inout}. Since $v_{-a} = v_{-h,1}$ or $v_{-a} = v_{-h,2}$ in this case then $v_a \in H_{\alpha}$ as well. 
Therefore:
$$
\begin{array}{l}
N_{H_{\alpha}}^+(v)  = \left\{v_j \mid j \in J\right\} = \left\{v_j \mid j \in J \setminus \left\{a,h\right\} \right\} \cup \left\{v_a,v_h \right\},\\
N_{H_{\alpha}}^-(v) = \left\{v_{-j} \mid j \in J'\setminus \left\{a\right\} \right\}\cup \left\{v_{-h,1},v_{-h,2}\right\}.
\end{array}
$$
By Corollary \ref{corol:inout}, it turns out that  $J \setminus \left\{a,h\right\} = J' \setminus \left\{ a \right\}$.
\end{itemize}
In both cases the conclusion can be derived easily. 
\end{itemize}
Therefore, the out-degree of $v$ is equal to its in-degree and the claim follows.
\end{proof}

From this, we can prove the subdigraphs $H_{\alpha}$ are diregular.

\begin{proposition}
\label{prop:0}
Any subdigraph $H_{\alpha}$ of a $(d,k)$-digraph $G$ is diregular.
\end{proposition}

\begin{proof} Let $v\in H_{\alpha}$ and $N^+(v)=\{v_1,v_2,\ldots,v_d\}$. Without loss of generality, let $v_1,\ldots,v_s \in H_{\alpha}$ and $v_{s+1},\ldots,v_d \not \in H_{\alpha}$. By Corollary \ref{corol1} it is enough to show that $d^-_{H_{\alpha}}(v_1) = s$ which in turn is the same as showing 
$$|N^-(v_1) \cap \overline{H}_{\alpha}| = d-s,
$$
where $\overline{H}_{\alpha}=G\setminus H_{\alpha}.$
Now let us introduce the following sets (see Figure \ref{connec-wv1}):
$$
W_{1}(j)=\{(w,v_1) \in E(G) \mid w \in T(v_j)\}, \quad 1\leq j \leq d.
$$

\begin{figure}[h!]
   \begin{center}
    \includegraphics[scale=0.75]{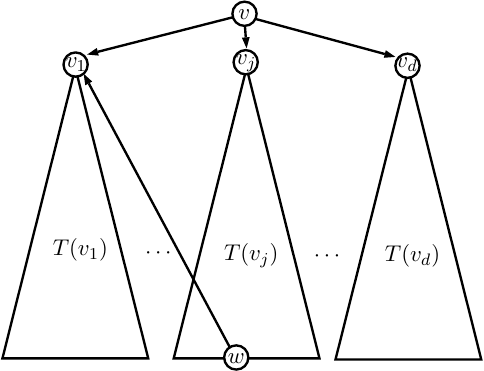}
    \caption{The connection between $w$ and $v_1$.}
    \label{connec-wv1}
    \end{center}
\end{figure}

Some key observations:
\begin{itemize}
    \item[i)] $|W_1(j)| \leq 2$ for every $j \in \left\{2,\ldots,d\right\}$ because otherwise there would be at most $N-1$ vertices in $G$, again, contradicting the definition of an almost Moore digraph.

\item[ii)] Moreover, if $|W_1(j)| = 2$ with $2 \leq j \leq d$ then $v_1 = r(v_j)$ implying $j \leq s$ because $r(H_{\alpha}) \subseteq H_{\alpha}$ since $H_{\alpha}$ is $(r,k)$-closed.

\item[iii)] If $|W_1(j)| = 0$ then $v=r(v)$ and $v_j=r(v_j)$ implying $j \leq s$ because $v \in H_{\alpha}$ together with the definition of $H_{\alpha}$.

\item[iv)] If $W_1(j_1) \cap W_1(j_2) \neq \emptyset$ for some $j_1\neq j_2$ then $T(v_{j_1}) \cap T(v_{j_2}) \neq \emptyset$. Thus we must have $T(v_{j_1}) \cap T(v_{j_2}) = \left\{r(v)\right\}$. Hence $j_1,j_2 \leq s$, because $v \in H_{\alpha}$. Moreover $r(H_{\alpha}) \subseteq H_{\alpha}$ by definition of $H_{\alpha}$. 

\item[v)] For any $w_j \in W_1(j)$, with $j$ such that $1 \leq j \leq d$, the shortest walks given by 
$$
v_j \to \cdots\to w_j \to v_1 \quad \mbox{ and }\quad v \to v_j \to \cdots \to w_j$$ 
have length at most $k$. Therefore, if $w_j \in H_{\alpha}$ then $v_j \in H_{\alpha}$ because both $v,w_j$ belong to $H_{\alpha}$ and $H_{\alpha}$ is $(r,k)$-closed. Similarly, if $v_j \in H_{\alpha}$ then $w_j \in H_{\alpha}$ because both $v_1,v_j \in H_{\alpha}$ and  $H_{\alpha}$ is $(r,k)$-closed.  
\end{itemize}
Therefore, by (i), (ii) and (iii) every $j$ with $s < j  \leq d$ must satisfy $|W_1(j)|=1$.  Then, consider 
$$
W_1(j) = \left\{w_{j} \right\}, \quad s<j\leq d.
$$
Thus, by (iv), $w_{j_1} \neq w_{j_2}$ for any two distinct values $j_1,j_2$ such that $s < j_1, j_2 \leq d$. By (v), $w_j \in \overline{H}_{\alpha}$. Therefore, we have that $|N^-(v_1) \cap \overline{H}_{\alpha}| \geq d-s$. 
On the other hand, by (v) if $w_j \in N^-(v_1) \cap \overline{H}_{\alpha}$ then $v_j \in \overline{H}_{\alpha}$, which is the same as saying $s < j \leq d$. From here $d-s \geq |N^-(v_1) \cap \overline{H}_{\alpha}|$.
Combining these results we reach to the conclusion that $|N^-(v_1) \cap \overline{H}_{\alpha}| = d-s$ as claimed. 
\end{proof}

\subsection{The subdigraphs $H_{\alpha}$ different from $C_k$ are almost Moore digraphs}

Now we show that any subdigraph $H_{\alpha}$ of a $(d,k)$-digraph, $H_{\alpha}\neq C_k$, is also a $(d',k)$-digraph, with $d'\leq d$.

\begin{lemma}
\label{lemm:1}
Suppose that $H_{\alpha}$ is $d'$-regular. Then for any vertex $u \in H_{\alpha}$ there exist $d'^j$ vertices of $H_{\alpha}$ at distance $j$ from $u$ in $G$ if $ j < k$ and $d'^k-1$ vertices of $H_{\alpha}$ at distance $k$ from $u$ in $G$. 
\end{lemma}

\begin{proof}
This claim follows easily applying induction. Indeed, if there are $d'^j$ vertices at distance $1 \leq j<k-1$ from $u$, by Proposition \ref{prop:0}, $H_{\alpha}$ is $d'$-regular, so there are $d' \cdot d'^j = d'^{j+1}$ vertices at distance $j+1$ from $u$. Therefore, in particular, there are $d'^{k-1}$ vertices at distance $k-1$ from $u$, implying that there are $d'\cdot d'^{k-1}-1 = d'^{k}-1$ vertices at distance $k$ from $u$ because $r(u) \in H_{\alpha}$ since $ord_r(u) = ord_r(r(u))$. 
\end{proof}

\begin{proposition}
\label{lemm:2}
Any subdigraph $H_{\alpha}$ of a $(d,k)$-digraph $G$ has diameter at most $k$. 
\end{proposition}

\begin{proof}
Let $u,v$ be vertices of $H_{\alpha}$. Since $H_{\alpha} \subseteq G$ the distance from $u$ to $v$ is at most $k$ in $G$, implying that there exists a walk $\pi$ connecting $u$ with $v$ of length at most $k$. Since $H_{\alpha}$ is $(r,k)$-closed by Proposition  \ref{kclosed} then $\pi$ is contained inside $H_{\alpha}$, too. Therefore, the distance from $u$ to $v$ inside $H_{\alpha}$ is at most $k$ as well. 
\end{proof}

\begin{corollary}
\label{cor3}
If $H_{\alpha} \neq C_k$ then it has diameter $k$.  
\end{corollary}

\begin{proof}
If $H_{\alpha} \neq C_k$ then there exists $d' \geq 2$ such that $H_{\alpha}$ is diregular of degree $d'$ by Proposition \ref{prop:0}. 
Now, let $u$ be any vertex from $H_{\alpha}$. From Lemma \ref{lemm:1} we know that there exist $d'^k-1 \geq d' -1$ vertices of $H_{\alpha}$ inside $N_k^+(u)$. Therefore, if $d' \geq 2$ there exists at least one vertex $v$ of $H_{\alpha}$ at distance exactly $k$ from $u$. Then $d_{H_{\alpha}}(u,v) \geq d_G(u,v) = k$ implying the conclusion. 
\end{proof}

Now, from previous results  we end up deducing the main result of this section.

\begin{theorem} Any subdigraph $H_{\alpha}$ of a $(d,k)$-digraph $G$ different from $C_k$ is a $(d',k)$-digraph, with $d'\leq d$.
\end{theorem}
\begin{proof}
   We know,  by Proposition \ref{prop:0}, that $H_{\alpha}$ is diregular with degree $d'\leq d$ and its diameter, from  Corollary 
    \ref{cor3}, is $k$. Then it is straightforward by Lemma \ref{lemm:1} 
    to derive that $H_{\alpha}$  has order $d'+d'^2+\cdots +d'^k$. Hence, it is a $(d',k)$-digraph.
\end{proof}

\section{Permutation cycle structure}
In this section we will study  certain permutation cycle structures that some
digraphs $H_{\alpha}$ can have.

\begin{definition}
  \label{crit}
  A permutation cycle structure $(m_1,\ldots,m_j,\ldots,m_N)$ is said to be $2$-critical if there exists a unique  subindex $\alpha > 1$ for which the subindices $j>1$ such that $m_j \neq 0$ are those of the form $j = 2^t\alpha$ for some $t \geq 0$.
\end{definition}

This definition is crucial because as we shall see at least one subdigraph $H_{\alpha}$ of a given almost Moore digraph has this type of permutation cycle. Note that if $\alpha$ is a prime number then $H_{\alpha}$ is either the subdigraph $C_k$ or it has 2-critical structure. 





\begin{proposition} 
\label{thm:1} 
Any $(d,k)$-digraph $G$ with self-repeats contains a subdigraph $H_{\alpha}$ which has $2$-critical structure.
\end{proposition}

\begin{proof}
  Since $G$ is different from $C_k$, it contains vertices of order $>1$. Consider the smallest $\alpha>1$ such that there exist a vertex of order $\alpha$. Then  the subdigraph $H_{\alpha}$ is also different from $C_k$ and according to the definition (\ref{Halpha}) its permutation cycle structure is 2-critical. 
\end{proof}

Now we take a look at a reasoning used by Baskoro, Cholily and Miller in \cite{enumerations} which is extremely useful to reach the final conclusion of this section. 

\begin{proposition} (\cite[Thm. 1]{enumerations})
\label{enum}
Let $v_0$ be a self-repeat of a $(d,k)$-digraph $G$. Let $S_1 = \left\{s_0,\ldots,s_t\right\}$ be the set containing all the orders of vertices in $N^+_1(v_0)$ and let $S_m$, $m\geq 2$, be the subset of values $\lcm(s_i,s_j)$, with $s_i, s_j \in S_{m-1}$. 
If $v\in N^+_j(v_0)$, $2\leq j \leq  k$, then $ord_r(v)\in S_j$.
\end{proposition}

Then, considering Proposition \ref{thm:1} together with the Theorem given in \cite{enumerations} we get the following result.

\begin{theorem}
\label{corol:1} Any $(d,k)$-digraph $G$ with self-repeats and 2-critical structure satisfies 
$$
ord_r(v) \leq d-1, \quad \forall v\in G.
$$
\end{theorem}
\begin{proof}
    Let $v_0$ be a self-repeat of $G$. Since $r(N^+_1(v_0))=N^+_1(r(v_0))$ (see \cite{b95}), we can consider the permutation $r$ on $N^+_1(v_0)$. As in Proposition \ref{enum}, let $S_1= \left\{s_0,s_1,\ldots,s_h\right\}$  be the set of orders of vertices in $N^+_1(v_0)$. We assume that  the permutation cycle structure of $G$ is $2$-critical, that is, there exists a value $\alpha$ such that $m_j \neq 0$ for $j > 1$ if $j = 2^t\alpha$, with $t \geq 0$. Then, $S_1$ contains the values  $2^{t_1}\alpha,2^{t_2}\alpha,\ldots,2^{t_h}\alpha$ with $t_1,\ldots,t_h \geq 0$ apart from the value $1$. But in this case, $\lcm(2^{t_a}\alpha,2^{t_b}\alpha) = 2^{\max(t_a,t_b)}\alpha$ and $\lcm(2^{t_c}\alpha,1) = 2^{t_c}\alpha$. 
Therefore, $S_m \subseteq S_1$ for all $m \geq 1$. Hence, it is clear that 
\begin{equation}
  \label{d-1}
d=1+\sum_{i=1}^hn_is_i,
\end{equation}
where $n_i$ is the multiplicity of $s_i$, the claim follows.
\end{proof}


\begin{corollary}
\label{maincor}
   Any $(d,k)$-digraph $G$ with self-repeats contains a $(d',k)$-subdigraph $H_{\alpha}$, $d'\leq d$, whose permutation cycle structure is $2$-critical with $m_j=0$ for $j>d'-1$. Moreover, $\alpha\mid (d'-1)$.
\end{corollary}
\begin{proof}
From Proposition \ref{thm:1} the digraph $G$ contains a $(d',k)$-subdigraph $H_{\alpha}$, $d'\leq d$, whose permutation cycle structure $(k,\ldots,m_j,\ldots)$ is $2$-critical. Then, by  Theorem \ref{corol:1}, $ord_r(v) \leq d'-1$, $\forall v\in H_{\alpha}$, implying that $m_j=0$, $\forall j\geq d'$. 
  In this case, equality (\ref{d-1}) can be rewritten as  $d'=1+\sum_{i=1}^hn'_i2^{t_1}\alpha$. Then, clearly $\alpha\mid (d'-1)$. 
  \end{proof}

\section{Matrix approach}

Let $G$ be a $(d,k)$-digraph with self-repeats and 
permutation cycle  structure
$
(k,\ldots,m_j,\ldots).
$
The characteristic polynomial of the matrix $\P$ associated to the permutation $r$ of repeats is given by
$$
 c(\P,x) = (x - 1)^k \prod_{m_j\neq 0}
  (x^j-1)^{m_j}.
$$
Its factorization in $\mathbb{Q}[x]$ in terms of
cyclotomic polynomials $\Phi_{i}(x)$, since $x^j-1=\prod_{i\mid j}\Phi_i(x)$, is:
\begin{equation}
\label{char:P}
  c(\P,x) = (x - 1)^{m(1)} \prod_{i\mid j,\; m_j\neq 0}
  \Phi_{i}(x)^{m(i)},
\end{equation} 
where $m(i) = \sum_{i|j} m_j$ represents the total
number of permutation cycles of order multiple of $i$. Thus the matrix $\P$ is diagonalizable with eigenvalues:
\begin{itemize}
  \item 
    $1$  with multiplicity $m(1)$;
  \item
    $\zeta_h^i$, $1\leq h\leq \varphi(i)$, primitive roots of unity of order $i$ with multiplicity $m(i)$. 
\end{itemize}
Therefore, the characteristic polynomial of the matrix $\J+\P$ (see \cite{b94}) is 
\begin{equation}
\label{char:J+P}
  c(\J+\P,x) = (x-(N+1))(x - 1)^{m(1)-1} \prod_{i\mid j,\; m_j\neq 0}
  \Phi_{i}(x)^{m(i)}.
\end{equation} 
Its eigenvalues are the same eigenvalues than the $\zeta_h^i$ of $\P$, except $1$ which has multiplicity $m(1)-1$, and the eigenvalue $N+1$ with multiplicity $1$.

From Lemma 1 in \cite{almost4-5}, the adjacency matrix $\A$ of $G$, which satisfies 
 $
 \I+\A+\cdots +\A^k=\J+\P,
 $
 is also diagonalizable with eigenvalues:
\begin{itemize}
 \item 
    $d$  with multiplicity $1$;
  \item
    $\lambda_h$, $1\leq h\leq m(1)-1$, roots of the factor $x^k+\cdots +x^2+x$; 
    \item
    $\alpha_{h,n}^i$, $1\leq n\leq m(i)$, roots of the factor $x^k+\cdots + x^2+x+1-\zeta_h^i$. 
\end{itemize}
 Note we can take a basis of eigenvectors of $\A$ which is as well a basis of eigenvectors of $\P$. More precisely, the eigenvectors of $\A$ of eigenvalues $\alpha_{h,n}^i$, $1\leq n\leq m(i)$, are eigenvectors of $\P$ of eigenvalue $\zeta_h^i$. And the eigenvectors of $\A$ of eigenvalues $d$ and $\lambda_h$, $1\leq h\leq m(1)-1$, are eigenvectors of $\P$ of eigenvalue $1$.

\subsection{Computing some traces of $\P^j\A^{\ell}$}
 
The traces of $\A^{\ell}$, $1\leq \ell\leq k$, are well-known and given in (\ref{traces1k}) and (\ref{tracek}). In order to study the traces of $\P^j\A^{\ell}$,  for  $j\geq 1$,
 we consider the sets  
$$
R_{\ell,j}(G) = \left\{ v\in V(G) \mid \text{there is a } r^j(v)\rightarrow v \text{ walk of length } \ell \right\},
$$ 
which are a generalisation of  $R_{\ell}(G)$ introduced in \cite{almost4-5}. Note that $R_{\ell}(G)=R_{\ell,1}(G).$

\begin{proposition} 
\label{Rell}
Let $(k,\ldots,m_{d-1},0,\ldots,0)$ be the permutation cycle structure of a $(d,k)$-digraph with self-repeats. If $\ell(d-1)<k+1$ then $R_{\ell,j}(G)=\emptyset$.
\end{proposition}
\begin{proof}
  The case $j=1$ follows from Proposition 3 in \cite{almost4-5}.  
  Since $r$ is a graph automorphism, given $\pi$ a path of length $\ell$ connecting $r^j(v)$ with $v$, $r^{tj}(\pi)$ with $t$ such that $1 \leq t \leq ord_r(v)$ are also paths of length $\ell$ connecting $r^{tj}(v)$ with $r^{(t-1)j}(v)$. More precisely, given such a path $\pi$, as long as $\ell t < k$, since $G$ does not contain cycles of length $< k$, concatenating the following directed paths of length $\ell$ gives another path of length $\ell t < k$ connecting $r^{tj}(v)$ with $v$: 
$$
r^{tj}(\pi) \to r^{(t-1)j}(\pi) \to \cdots \to r^{j}(\pi) \to \pi.
$$
Therefore, since $ord_r(v) \leq d-1$, then when $\ell(d-1) < k+1$ we have $|R_{\ell,j}(G)| = 0$.
\end{proof}

\begin{corollary} 
\label{prop:1}
Let $\A$ be the adjacency matrix of a $(d,k)$-digraph with permutation matrix $\P$ and permutation cycle structure $(k,\ldots,m_{d-1},0,\ldots,0)$. Then, for $j>0$, 
$$\Tr(\P^j\A^{\ell}) = 0, \quad 1\leq \ell < \dfrac{k+1}{d-1}.$$
\end{corollary}
\begin{proof}
The case $j=1$ follows from Corollary 1 in \cite{almost4-5}. For $j>1$, since the cardinality of the sets $R_{\ell,j}(G)$ have the following interpretation: 
$$
|R_{\ell,j}(G)|=\Tr(\P^j\A^{\ell}),
$$
the claim follows directly by Proposition \ref{Rell}.
\end{proof}

Using this result we deduce the next proposition:

\begin{proposition} 
\label{prop:2}
Let $G$ be  a $(d,k)$-digraph with self-repeats and permutation cycle structure
$(k,\ldots,m_{d-1},0,\ldots,0)$. Then, the eigenvalues  $\lambda_h$, $1\leq h\leq m(1)-1$, of its adjacency matrix $\A$ which are roots of $x^k+\cdots +x^2+x$ satisfy:
\begin{equation}
\label{eigenv}
0=d^{\ell}+ \sum_{h=1}^{m(1)-1}\lambda_h^{\ell}, \quad 1\leq \ell <\dfrac{k+1}{d-1}.
\end{equation}
\end{proposition}
\begin{proof} According to the eigenvalues of $\A$ given at the beginning of this section, for each $\ell$ satisfying $1\leq \ell <\dfrac{k+1}{d-1}$
it turns out that 
$$
0=\Tr(\A^{\ell})=d^{\ell}+ \sum_{h=1}^{m(1)-1}\lambda_h^{\ell}+ \sum_{m(i)\neq 0}\Big(\sum_{1\leq h\leq \varphi(i)}\Big(\sum_{1\leq n\leq m(i)} \alpha_{h,n}^i\Big)\Big).
$$
Taking into account the eigenvalues of $\P$, together with Corollary \ref{prop:1}, we have
$$
0=\Tr(\P\A^{\ell})=d^{\ell}+ \sum_{h=1}^{m(1)-1}\lambda_h^{\ell}+ \sum_{m(i)\neq 0}\Big(\sum_{1\leq h\leq \varphi(i)}\zeta_h^i\Big(\sum_{1\leq n\leq m(i)} \alpha_{h,n}^i\Big)\Big).
$$
Consider now the greatest integer $s$ such that $m_s\neq 0$. Again from Corollary \ref{prop:1}
$$
0=\Tr(\A^{\ell}) + \Tr(\P\A^{\ell})+\cdots + \Tr(\P^{s-1}\A^{\ell})
=\Tr((\I+\P+\cdots + \P^{s-1})\A^{\ell}).
$$
Therefore
$$
0=s\Big(d^{\ell}+ \sum_{h=1}^{m(1)-1}\lambda_h^{\ell}\Big)+ 
\sum_{m(i)\neq 0}\Big(\sum_{1\leq h\leq \varphi(i)}(1+\zeta_h^i+\cdots +(\zeta_h^i)^{s-1})\Big(\sum_{1\leq n\leq m(i)} \alpha_{h,n}^i\Big)\Big).
$$
Since $1+\zeta_h^i+\cdots +(\zeta_h^i)^{s-1}=0$, for each $\zeta_h^i$ root of unity or order a divisor $s$, the claim follows. 
\end{proof}

\section{Nonexistence of $(d,k)$-digraphs with self-repeats having large diameter $k$ with respect to the degree $d$}
Now we will see, taking into account previous results,
that $(d,k)$-digraphs do not exist when the diameter $k$ is large enough with respect to the degree $d$.

In order to prove this, note that as in \cite{almost4-5}, the eigenvalues $\lambda^{\ell}_h$, $1\leq h\leq m(1)-1$, of $\A$ satisfying (\ref{eigenv}) 
are roots of the factor
$x^k+\cdots + x^2+x$, 
which can be expressed as 
$$x\prod_{n>1,\, n\mid k}\Phi_n(x).
$$
Then, it turns out that $x$ and $\Phi_n(x)$, with $n>1$ and $n\mid k$, are factors of the characteristic polynomial of $\A$ with multiplicities $a_0$ and $a_n$, respectively, satisfying 
$$
a_0+\sum_{n>1,\, n\mid k}\varphi(n)a_n=m(1)-1.
$$
Therefore 
\begin{equation}
\label{ram}
\sum_{h=1}^{m(1)-1} \lambda_h^\ell=\sum_{n > 1,\,n\mid k}a_nS_\ell(\Phi_n(x)), 
\end{equation}
where  $S_\ell(a(x))$ denotes the sum of the $\ell$th powers of all roots of $a(x)$. 

The sums $S_\ell(\Phi_n(x))$, called \emph{Ramanujan sums}, can be computed as (see \cite{HW}) 
\begin{equation}
\label{ramanujan}
S_\ell(\Phi_n(x)) = \sum_{j \mid \gcd(n,\ell)}\mu\left(\frac{n}{j}\right)j,\end{equation}

\noindent where $\mu(n)$ is the M\"obius function.

\begin{proposition}
\label{prime}
Given a pair of integers $(d,k)$, $d>1$, $k>1$, such that 
there exists a prime $\ell$ with $\ell(d-1) < k+1$ and $\gcd(\ell,k) = 1$, then there cannot exist any $(d,k)$-digraph with self-repeats and permutation cycle structure $(k,\ldots,m_{d-1},0,\ldots,0)$.
\end{proposition}

\begin{proof}
From Proposition \ref{prop:2}, together with expression (\ref{ram}), if such digraph exists then its adjacency matrix $A$ would satisfy:
\begin{equation}
\label{eq:2}
0=d^{\ell}+\sum_{n> 1,\,n\mid k}a_nS_\ell(\Phi_n(x)), \quad 1\leq \ell <\dfrac{k+1}{d-1}.
\end{equation}
If $\ell$ and $k$ are relatively primes then for every $n\mid k$ we have  by (\ref{ramanujan}):
 $$
  S_\ell(\Phi_n(x))=\mu(n).
  $$
  Thus, if there exists a prime $\ell$ such that
    $\gcd(\ell,k)=1$ and $1<\ell<\frac{k+1}{d-1}$,
  then $S_\ell(\Phi_n(x))=S_1(\Phi_n(x))$ for all $n$ with $n\mid k$,
  which would imply that
  $$
  d^{\ell}-d=0,
  $$
  but this is not possible except $d=1$ or $\ell=1$.
\end{proof}

The existence of a prime $\ell$ satisfying Proposition \ref{prime} allows us to prove the nonexistence of $(d,k)$-digraphs with self-repeats when $(d,k)$ fulfills the following conditions. 

\begin{theorem}
\label{thmmain}
For $d\geq 6$, there do not exist  $(d,k)$-digraphs with self-repeats  when either 
\begin{itemize}
\item 
$k \geq 2(d-1)$ and $k$ odd, or
   \item 
$k \geq 2(d-1)^2$ and $k$ even.
   \end{itemize}
\end{theorem}
\begin{proof}
We first prove that in these conditions there are no $(d,k)$-digraphs with self-repeats and permutation cycle structure $(k,\ldots,m_{d-1},0,\ldots,0)$. Indeed, according to Proposition \ref{prime}, it is enough to find a prime value $\ell$ such that
\begin{equation}
  \label{cond}
\gcd(\ell,k)=1 \text{ and }
1<\ell< \dfrac{k+1}{d-1}.
\end{equation}
Consider the distinct consecutive prime numbers until $(k+1)/(d-1)$: 
\begin{equation}
2 = p_1 < p_2 < \cdots < p_s < \frac{k+1}{d-1} \leq p_{s+1}.
\label{seqprime}
\end{equation}
If there exists $j$ with $1 \leq j \leq s$ such that $p_j \nmid k$ from Proposition \ref{prime} it turns out there are no $(d,k)$-digraphs for such $d$ and $k$. In particular, if $k$ is odd, taking $\ell=2$, we deduce that for $k\geq 2(d-1)$ there are no $(d,k)$-digraphs with structure $(k,\ldots,m_{d-1},0,\ldots,0)$.  
Otherwise, $k$ is even and 
$$p_1p_2\ldots p_s \mid k.$$
Now, assume  $k\geq 2(d-1)^2$. Since there exists at least a prime number between $p_s$ and $2p_s$, it turns out that
\begin{equation}
2(d-1)< \dfrac{k+1}{d-1} \leq p_{s+1} < 2p_s.
\label{quo}
\end{equation}
Then $d-1< p_{s}$ and 
$$k < 2p_s(d-1) < 2p_s^2.$$
If $s \geq 4$, since $p_{s-1}<p_s<2p_{s-1}$, we have
$$ 
2p_s^2 \leq p_{s-3}(2p_{s-1})p_{s}\leq  p_{s-3}p_{s-2}p_{s-1}p_{s}   \leq  k < 2p_s^2.
$$
Thus, we must have $s \leq 3$ and, taking into account (\ref{seqprime}) and (\ref{quo}), we get 
$$2(d-1)^2 \leq k < p_4(d-1)=7(d-1).$$
This implies $ 2(d-1) < 7$, which is  a contradiction since $d \geq 6$. Therefore, if $k$ is even it turns out that for $k\geq 2(d-1)^2$ there are no $(d,k)$-digraphs with self-repeats and $m_j=0$, $\forall j\geq d$. 

Assume now there exists a $(d,k)$-digraph with self-repeats $G$ such that the pair $(d,k)$ satisfies the conditions of the claim. According to Corollary (\ref{maincor}), $G$ contains a $(d',k)$-subdigraph $H_{\alpha}$, $d'\leq d$, whose permutation cycle structure  is 2-critical  with $m_j=0$ for all $j>d'-1$. Since either $k \geq 2(d-1)\geq 2(d'-1)$ if $k$ odd or
 $k \geq 2(d-1)^2\geq 2(d'-1)^2$ if $k$ even, $H_{\alpha}$ can not exist. Therefore, neither does $G$.          
\end{proof}



Combining
Theorem \ref{thmmain} together with Remark \ref{rem}, we can prove the nonexistence of $(d,k)$-digraphs with self-repeats for degrees $d$, $6\leq d\leq 12$. 

\begin{corollary}
  There are no $(d,k)$-digraphs, $k> 2$, with self-repeats for
  $$d\in\{6,7,8,9,10,11,12\}.$$
\end{corollary}
\begin{proof}
  By Theorem \ref{thmmain}, $(d,k)$-digraphs with self-repeats do not exist when $k\geq 2(d-1)^2$. Since these digraphs contain a subdigraph $H_{\alpha}$ whose permutation cycle structure satisfies $m_i=0$, for $i>d-1$ (see Corollary \ref{maincor}),  it is straightforward  to show they neither exist for degrees $d$, $6\leq d\leq 11$, and $2<k<2(d-1)^2$.
  Indeed, the polynomials $F_{i,k}(x)$, $2<i\leq 12$ and $2<k<200$, factorize in $\mathbb Q[x]$ according to the conjecture, which is checked computationally after Remark \ref{rem}. Therefore, by \cite{cggmm14} $(d,k)$-digraphs do not exist for degrees $d$ such  that $5\leq d-1\leq 10$ and $k>2$.

 For degree $d=12$, again by Remark \ref{rem} such digraphs do not exist  when $2<k<200$ and by Theorem \ref{thmmain} when $k\geq 2(d-1)^2=242$.  Besides, if $k$ is odd they do exist neither when $k\geq 2(d-1)=22$. Then, it is enough to see that the $(12,k)$-digraphs do not exist  for $k$ even in the interval $[200, 240]$. But for this range of $k$ one can find a prime $\ell$ satisfying conditions (\ref{cond}).
\end{proof}

Finally, as a summary, Figure \ref{values}  shows the values of the degre $d$ and the diameter $k$ it is proven the existence or nonexistence of $(d,k)$-digraphs with self-repeats.

\begin{figure}[htb]
\begin{center}
\includegraphics[scale=0.7]{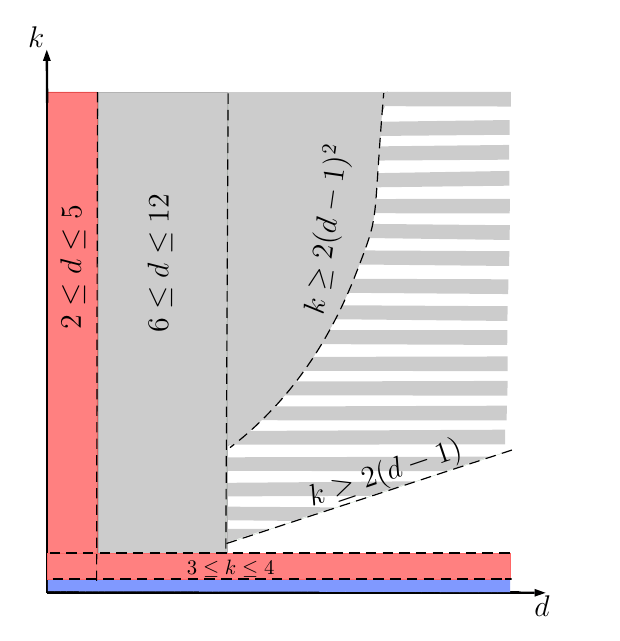}
\caption{Values of $d$ and $k$ for the nonexistence of $(d,k)$-digraphs.}
\label{values}
\end{center}
\end{figure}

 We have drawn in blue the line $k=2$ for which we know the existence of $(d,2)$-digraphs. The area in red corresponds to the values $d\in[2,6]$ and $k\in[3,4]$, where we already knew the nonexistence of $(d,k)$-digraphs with self-repeats. And the area determined by $6\leq d\leq 12$, as well as the area $k\geq 2(d-1)^2$ together with lines $k$ odd between $[2(d-1),2(d-1)^2]$, all of them drawn in grey, are the new values for which the nonexistence has also been proven. 




\begin{thebibliography}{1}

    \bibitem{enumerations} E.T. Baskoro, Y.M. Cholily, M. Miller,
  Enumerations of vertex orders of almost Moore digraphs with 
  selfrepeats, {\em Discrete Math.\/} 308(1): 123--128, 2008.

\bibitem{baskoro96} E.T. Baskoro, M. Miller and  J. Plesn\'{\i}k,
On the structure of digraphs with order close to the Moore bound,
{\em Graphs Combin.\/} 14: 109--119, 1998.

\bibitem{b94} E.T. Baskoro, M. Miller, J. Plesn\'{\i}k and  \v{S}. Zn\'{a}m,
Regular digraphs of diameter 2 and maximum order,
{\em Australas. J.  Combin.\/} 9: 291--306, 1994.

\bibitem{b95} E.T. Baskoro, M. Miller, J. Plesn\'{\i}k and  \v{S}. Zn\'{a}m,
Digraphs of degree 3 and order close to the Moore bound,
{\em  J. Graph Theory\/} 20: 339--349, 1995.

\bibitem{baskoro973} E.T. Baskoro, M. Miller, J. \v{S}ir\'a\v{n} and M. Sutton, Complete
characterisation of almost Moore digraphs of degree three,
{\em J. Graph Theory\/} 48(2): 112--126, 2005.

\bibitem{Bosak} J. Bos\'ak,
  Directed graphs and matrix equations, {\em Math. Slovaca} 28: 189--202, 1978.

\bibitem{bridges80} W.G. Bridges and S. Toueg, On the impossibility of
directed Moore graphs, {\em J. Combin. Theory Ser. B\/} 29: 
339--341, 1980.

\bibitem{cholily} Y.M. Cholily, 
A conjecture on the existence of almost Moore digraphs,  
{\em Adv. Appl. Discrete Math.} 8(1): 57--64, 2011.

\bibitem{CBU} Y.M. Cholily, E.T. Baskoro and S. Uttunggadewa,
Some conditions for the existence of $(d,k)$-digraphs,  {\em LNCS} 3330: 87--93, 2005.

\bibitem{cggmm14}
J. Conde, J. Gimbert, J. Gonz\'alez, M. Miller and J. Miret, On the nonexistence of almost Moore digraphs, {\em European J. Combin.\/}, 39: 170--177, 2014.

\bibitem{cggmm08} J. Conde, J. Gimbert, J. Gonz\'alez, J. Miret and R. Moreno,
Nonexistence of almost Moore digraphs of diameter three, {\em Electron. J. Combin.\/} 15 \#R87, 2008.

\bibitem{cggmm13} J. Conde, J. Gimbert, J. Gonz\'alez, J. Miret and R. Moreno,
Nonexistence of almost Moore digraphs of diameter four, {\em Electron. J. Combin.\/} 20(1) \#P75, 2013.

\bibitem{FJ}
S. Filipovski and R. Jajcay, On the non-existence of families of $(d,k,\delta)$ digraphs containing only selfrepeat vertices, {\em Linear Algebra and its Applications\/} 563: 302--312, 2019. 

\bibitem{fiol83} M.A. Fiol, I. Alegre and J.L.A. Yebra, Line digraphs
iterations and the $(d,k)$ problem for directed graphs,
{\em Proc. 10th Int. Symp. Comput. Arch.\/},  174--177, 1983.

\bibitem {g99}
 J. Gimbert, On the existence of $(d,k)$-digraphs, {\em Discrete Math}. 197/198: 375--391, 1999.

\bibitem{gimbert01} J. Gimbert, 
 Enumeration of almost Moore digraphs of diameter two, {\em Discrete Math.\/}, 231: 177--190, 2001.

\bibitem{HW} G.H. Hardy and E.M. Wright,
 {\em An Introduction to the Theory of Numbers\/}, Clarendon Press, 1979.

\bibitem{almost4-5} N. L\'opez, A. Messegu\'e and J. Miret,
  Nonexistence of almost Moore digraphs, {\em Electron. J. Combin.\/}, 30(1)
 \#P156,  2023.

\bibitem{miller92} M. Miller and I. Fris, Maximum order digraphs for
diameter 2 or degree 2, {\em Graphs, matrices, and designs}, 
Lecture Notes in Pure and Appl. Math. 139: 269--278, 1993.

\bibitem{diregular} M. Miller, J. Gimbert, J. \v{S}ir\'a\v{n} and Slamin, Almost Moore digraphs are diregular, {\em Discrete Math.\/} 218: 265--270, 2000.

\bibitem{ms05} 
 M. Miller and I. \v{S}ir\'a\v{n}, Moore graphs and beyond: A survey,
 {\em Electron. J. Combin}, DS14, 2005.

 \bibitem{mrs} 
   J. Miret, R. Simanjuntak and S. Sim\'on,
   On the permutation cycle structures for almost Moore digraphs of degrees 4 and 5, 
 {\em Discrete Math. Lett.} 12: 189-195,, 2023.

\bibitem{PZ74} J. Plesn\'{\i}k and \v{S}. Zn\'{a}m,
Strongly geodetic directed graphs,
{\em Acta Fac. Rerum Natur. Univ. Comenian. Math.} 29: 29--34, 1974.

\bibitem{fixedpoints}
A.A. Sillasen, Subdigraphs of almost Moore digraphs induced by fixpoints of an automorphism, {\em Electron. J. Graph Theory Appl.\/} 3(1): 1--7, 2015.


\end{thebibliography}
\end{document}